\documentclass[a4paper,11pt,intlimits,oneside]{amsart}
\usepackage{amsmath}
\usepackage{amsthm}
\usepackage{latexsym}
\usepackage{amssymb}
\usepackage{xcolor}
\numberwithin{figure}{section}
\def\R{{\mathbb R}}
\def\C{{\mathbb C}}
\def\T{{\mathbb T}}
\def\Z{{\mathbb Z}}

\def\s{\vskip 0.25cm\noindent}
\def\build#1_#2^#3{\mathrel{
\mathop{\kern 0pt#1}\limits_{#2}^{#3}}}
\def\td_#1,#2{\mathrel{\mathop{\build\longrightarrow_{#1\rightarrow #2}^{}}}}

\newtheorem{theo}{Theorem}
\newtheorem{corollary}{Corollary}
\newtheorem{proposition}{Proposition}
\newtheorem{Lemma}{Lemma}
\newtheorem{remark}{Remark}

\begin{document}
\title[Spectral inverse problems]{Spectral inverse problems for compact Hankel operators}
\author{Patrick G\'erard}
\address{Universit\'e Paris-Sud XI, Laboratoire de Math\'ematiques
d'Orsay, CNRS, UMR 8628, et Institut Universitaire de France} \email{{\tt Patrick.Gerard@math.u-psud.fr}}

\author[S. Grellier]{Sandrine Grellier}
\address{F\'ed\'eration Denis Poisson, MAPMO-UMR 6628,
D\'epartement de Math\'ematiques, Universit\'e d'Orleans, 45067
Orl\'eans Cedex 2, France} \email{{\tt
Sandrine.Grellier@univ-orleans.fr}}


\date{January 24, 2012}

\begin{abstract}
Given two arbitrary sequences $(\lambda_j)_{j\ge 1}$ and $(\mu_j)_{j\ge 1}$ of real numbers satisfying
$$|\lambda_1|>|\mu_1|>|\lambda_2|>|\mu_2|>\dots>\vert \lambda _j\vert >\vert \mu _j\vert  \to 0\ ,$$
we prove that there exists a unique sequence $c=(c_n)_{n\in\Z_+}$, real valued,  such that the  Hankel operators $\Gamma_c$ and  $\Gamma_{\tilde c}$  of symbols $c=(c_{n})_{n\ge 0}$ and $\tilde c=(c_{n+1})_{n\ge 0}$  respectively, are selfadjoint compact operators on $\ell^2(\Z _+)$ and have  the sequences $(\lambda_j)_{j\ge 1}$ and $(\mu_j)_{j\ge 1}$ respectively as non zero eigenvalues.  Moreover, we give an explicit formula for $c$ and we describe the kernel of $\Gamma_c$ and of $\Gamma_{\tilde c}$  in terms of the sequences $(\lambda_j)_{j\ge 1}$ and $(\mu_j)_{j\ge 1}$. More generally, given two arbitrary sequences $(\rho _j)_{j\ge 1}$ and $(\sigma _j)_{j\ge 1}$ of positive numbers satisfying
$$\rho _1>\sigma _1>\rho _2>\sigma _2>\dots> \rho  _j> \sigma  _j  \to 0\ ,$$
we describe the set of sequences  $c=(c_n)_{n\in\Z_+}$ of complex numbers such that the 
 Hankel operators $\Gamma_c$ and  $\Gamma_{\tilde c}$ are compact  on $\ell ^2(\Z _+)$ and have sequences $(\rho _j)_{j\ge 1}$ and $(\sigma _j)_{j\ge 1}$ respectively as non zero singular values.
\end{abstract}
\thanks {The second author acknowledges the support
of the  ANR project AHPI
(ANR-07-BLAN-0247-01).}
\maketitle

\section{Introduction}

Let  $c=(c_n)_{n\ge 0}$ be a sequence of complex numbers. The Hankel operator $\Gamma _c$ of symbol $c$ is formally defined on $\ell ^2(\Z _+)$ by
$$\forall x=(x_n)_{n\ge 0}\in\ell^2(\Z_+)\ ,\ \Gamma _c(x)_n=\sum _{p=0}^\infty c_{n+p}x_p\ .$$
These operators frequently appear in operator theory and in harmonic analysis, and we refer to the books by Nikolskii \cite{N} and  Peller \cite{P} for 
an introduction and their basic properties. By a well known theorem of Nehari \cite{Ne},  $\Gamma _c$ is well defined and bounded on $\ell^2(\Z_+)$ if and only if there exists
a function $f\in L^\infty (\T )$ such that $\forall n\ge 0, \hat f(n)=c_n$, or equivalently if the Fourier series $u_c=\sum_{n\ge 0} c_n {\rm e}^{inx }$ belongs to the space $BMO(\T )$ of bounded mean oscillation functions.  Moreover, by a well known result of Hartman  \cite{Ha}, $\Gamma_c$ is compact if and only if there exists a continuous function $f$ on $\T $ such that $\forall n\ge 0, \hat f(n)=c_n$, or equivalently if  $u_c$ belongs to the space $VMO(\T )$ of vanishing mean oscillation functions. 
Assume moreover that the sequence  $c$  is  real valued. Then $\Gamma_c$ is selfadjoint and compact, so it admits a sequence of non zero eigenvalues $(\lambda_j)_{j\ge 1}$, tending to zero. A natural inverse spectral problem is the following: {\sl given any sequence $(\lambda_j)_{j\ge 1}$, tending to zero, does there exist a compact selfadjoint Hankel operator $\Gamma_c$ having this sequence as non zero eigenvalues,
repeated according to their multiplicity?}

A complete  answer to this question can be found in the literature as a consequence of a more general theorem by Megretskii, Peller, Treil \cite{MPT} characterizing selfadjoint operators  which are unitarily equivalent to bounded Hankel operators. Here we state  the part of their result which concerns the compact operators.

\begin{theo}[Megretskii, Peller, Treil \cite{MPT}]
Let $\Gamma$ be a compact, selfadjoint operator on a separable Hilbert space. Then $\Gamma$ is unitarily equivalent to a Hankel operator if and only if the following conditions are satisfied
\begin{enumerate}
\item Either $\ker (\Gamma)=\{0\}$ or ${\rm dim}\ker (\Gamma)=\infty$; 
\item For any $\lambda \in \R^*$, $|{\rm dim}\ker (\Gamma-\lambda I)-{\rm dim}\ker (\Gamma+\lambda I)|\le 1.$
\end{enumerate}
\end{theo}

As a consequence of this theorem, any sequence of real numbers  with distinct absolute values and converging to $0$ is the sequence of the non zero eigenvalues of some compact selfadjoint Hankel operator.

In this paper, we are interested in finding additional constraints  on the operator $\Gamma_c$ which give rise to uniqueness of $c$. With this aim in view, we introduce  the shifted Hankel operator $\Gamma_{\tilde c}$, where  $\tilde c_n:=c_{n+1}$ for all $n\in \Z_+$.  If we denote by $(\lambda_j)_{j\ge 1}$ the sequence of non zero eigenvalues of $\Gamma_c$ and by $(\mu_j)_{j\ge 1}$ the sequence of non zero eigenvalues of $\Gamma_{\tilde c}$, one can check --see below-- that 
$$|\lambda_1|\ge |\mu _1|\ge |\lambda_2|\ge|\mu _2|\ge \dots\ge \dots \to 0\ .$$

Our result reads as follows.

\begin{theo}\label{HankelautoE}
 Let $( \lambda _j)_{j\ge 1}$ , $( \mu _j)_{j\ge 1}$ be two sequences of real numbers tending to zero so that $$\vert \lambda _1\vert >\vert \mu_1\vert >\vert \lambda  _2\vert >\vert \mu_2\vert >... >\dots \to 0\ .$$ 
There exists a unique real valued sequence $c=(c_n)$   such that $\Gamma _c$ and $\Gamma _{\tilde c}$ are compact selfadjoint operators, 
  the sequence of non zero eigenvalues of
$\Gamma _c$
is   $(\lambda _j)_{  j\ge 1},$ 
and the sequence of non zero eigenvalues of 
$ \Gamma _{\tilde c}$
is   $ (\mu_j)_{  j\ge 1}.$

Furthermore,  the kernel of $\Gamma_c$ is reduced to zero if and only if the following conditions hold,
\begin{equation}
\sum_{j=1}^\infty \left( 1-\frac{\mu_j^2}{\lambda_j^2}\right)=\infty,\; \qquad\;\sup_N\frac 1{\lambda_{N+1}^2}\prod_{j=1}^N\frac{\mu_j^2}{\lambda_j^2}=\infty.\end{equation}
Moreover, in that case, the kernel of $\Gamma_{\tilde c}$ is also reduced to $0$.

\end{theo}

In complement to the above statement, let us mention that   an explicit formula for $c$ is available, as well as an explicit description of the kernel of $\Gamma_c$ when it is non trivial --- see Theorems \ref{TheoHomeo} and \ref{KerHu} below. 

Theorem \ref{HankelautoE} is in fact a consequence of a more general result concerning the singular values of  non necessarily selfadjoint compact Hankel operators. Recall  that the singular values of a bounded operator $T$ on a Hilbert space $\mathcal H$, are given by the following min-max formula. For every $m\ge 1$, denote by $\mathcal F _m$ the set of linear subspaces of $\mathcal H$ of dimension at most $m$.  The $m$-th singular value of $T$ is given by

\begin{equation}\label{minmax}s_m(T)=\min_{F\in \mathcal F_{m-1}}\max_{f\in F^\perp, \Vert f\Vert=1}\Vert T(f)\Vert.\end{equation}

In this paper, we construct a homeomorphism between some set of symbols $c$ and the singular values of $\Gamma_c$ and $ \Gamma_{\tilde c}$ up to the choice of an element in an infinite dimensional torus. 

In order to state this general result we complexify and reformulate  the problem in the Hardy space. We identify $\ell^2(\Z_+)$ with 
$$L^2_+(\T)=\{ u\; :\;  u=\sum _{n=0}^\infty \hat u(n)\, {\rm e}^{inx }\ ,\ \sum _{n=0}^\infty |\hat u(n)|^2<+\infty \ \}$$ and we denote by $\Pi$ the orthogonal projector from $L^2(\T)$ onto $L^2_+(\T).$

Here and in the following, for any space of distributions $E$ on $\T$, the notation $E_+$ stands for the subspace of $E$ consisting of those elements $u$ of $E$ such that $\hat u(n)=0$ for every $n<0$, or equivalently  which can be holomorphically extended to the unit disc. In that case, we will still denote by $u(z)$ the value of this holomorphic extension at the point $z$ of the unit disc.

We endow $L^2_+(\T)$ with the scalar product 
$$(u|v):=\int _{\T}u\overline v\, \frac{dx }{2\pi }$$
and with the associated symplectic form
$$\omega (u,v)=\,{\rm Im}(u|v)\ .$$
For $u$ sufficiently smooth, we define a $\C$-antilinear operator  on $L^2_+$ by
$$H_u(h)=\Pi (u\overline h)\ ,\ h\in L^2_+\ .$$
If $u=u_c$, $$\widehat{H_u(h)}(n)=\Gamma_c(x)_n\ ,\ x_p:=\overline{\hat h(p)}\ .$$
Because of this equality, $H_u$ is called the Hankel operator of symbol $u$. Similarly, $\Gamma_{\tilde c}$ corresponds to the operator $K_u=H_uT_z$ where $T_z$ denotes multiplication by $z$.
Remark that by definition $H_u=H_{\Pi(u)}$. In the following, we always consider holomorphic  symbols $u=\Pi(u)$.

As stated before, by the  Nehari theorem (\cite{Ne}), $H_u$ is well defined and bounded on $L^2_+(\T)$ if and only if $u$ belongs to  $\Pi(L^\infty(\T))$ or to $BMO_+(\T)$.  Moreover, by the Hartman theorem (\cite{Ha}), it is a compact operator if and only if $u$ is the projection of a continuous function on the torus,  or equivalently if and only if it belongs to $VMO_+(\T)$ with equivalent norms.
Furthermore, remark  that this operator $H_u$ is selfadjoint as an antilinear operator in the sense that for any $h_1,h_2\in L^2_+$, $$(h_1\vert H_u(h_2))=(h_2\vert H_u(h_1)).$$ 

A crucial property of Hankel operators is that $H_uT_z=T_z^*H_u$ so that, in particular,
\begin{equation}\label{K_u^2}K_u^2=H_u^2-(\cdot\vert u)u.\end{equation}
 Assume $u\in VMO_+(\T)$ and denote by $(\rho_j)_{j\ge 1}$ the sequence of singular values of $H_u$ labelled according to the min-max formula (\ref{minmax}). Since, via the Fourier transform,  $H_{u}^2$ identifies to $\Gamma_c\Gamma_c^*$ with $c=\hat u$, $(\rho_j)_{j\ge 1}$ is also the sequence of singular values of $\Gamma _{\hat u}$.
 Similarly,  $K_u$ is a compact, so it  has a sequence $(\sigma_j)_{j\ge 1}$ of singular values tending to $0$, which are the singular values of $\Gamma _{\tilde c}$, since $K_u^2$ identifies to $\Gamma _{\tilde c}\Gamma _{\tilde c}^*$. From Equality (\ref{K_u^2}) and the min-max formula (\ref{minmax}), one obtains
$$\rho_1\ge \sigma_1\ge \rho_2\ge \sigma_2\ge \dots\ge \dots \to 0.$$
We denote by $VMO_{+,{\rm gen}}$ the set of $u\in VMO_+(\T)$ such that $H_u$ and $K_u$ admit only simple singular values with strict inequalities, or equivalently such that $H_u^2$ and $K_u^2:=H_u^2-(\cdot\vert u)u$ admit only simple positive eigenvalues $\rho_1^2>\rho_2^2>\dots >\dots \to 0$ and $\sigma_1^2>\sigma_2^2>\dots>\dots \to 0$ so that $$\rho_1^2>\sigma^2_1>\rho_2^2>\sigma^2_2>\dots>\dots \to0.$$
For any integer $N$, we denote by $\mathcal V(2N)$ the set of symbol $u$ such that the rank of $H_u$ and the rank of $K_u$ are both equal to  $N$. By a theorem of Kronecker (see \cite{Kr}), $\mathcal V(2N)$ is a complex manifold of dimension $2N$ consisting of rational functions. One can consider as well  the set $ {\mathcal V}(2N-1)$ of symbols $u$ such that $H_u$ is of rank $N$ and $K_u$ is of rank $N-1$. It defines a complex manifold of rational functions of complex dimension $2N-1$.

By the arguments developed in \cite{GG}, it is straightforward to verify that $VMO_{+,\rm gen}$ is  a dense $G_\delta$ subset of $VMO_+(\T)$.
Indeed, let us consider the set $\mathcal U_N$ which consists of functions $u\in VMO_+(\T)$ such that the $N$ first eigenvalues  of $H_u^2$ and of $K_u^2$ are simple. 
This set is obviously open in $VMO_+(\T)$. Moreover, in Lemma 4 of \cite{GG}, it is proved that $\mathcal U_N\cap \mathcal V(2N):=\mathcal V(2N)_{{\rm gen}}$ is a dense open subset of $\mathcal V(2N)$.
Now any element $u$ in $VMO_{+}$ may be approximated by an element in $\mathcal V(2N')$, $N'>N$, which can be itself approximated by an element in $\mathcal V(2N')_{\rm gen}\subset \mathcal U_N$, since $N'\ge N$. Eventually, $VMO_{+,\rm{gen}}$ is the intersection of the $\mathcal U_N$'s which are open and dense, hence $VMO_{+,\rm{gen}}$ is a dense $G_\delta$ set.\s
 Let $u\in  VMO_{+,\rm{gen}}$. Denote by  $((\rho_j)_{  j\ge 1}$ the singular values of $H_u$ and by $(\sigma_j)_{j\ge 1}$ the singular values of $K_u$.
Using the antilinearity of $H_u$ there exists an orthonormal family $(e_j)_{j\ge 1}$ of the range of $H_u$ such that
$$H_u(e_j)=\rho _je_j\ ,\  j\ge 1.$$
Notice that  the orthonormal family is determined by $u$ up to a change of sign on 
some of the $e_j$. We claim that $(1\vert e_j)\neq 0$. Indeed, if $(1\vert e_j)=0$ then $(u\vert e_j)=\rho _j(e_j\vert 1)=0$ and, in view of (\ref{K_u^2}),  $\rho_j^2$ would be an eigenvalue of $K_u^2$, which contradicts the assumption.
Therefore we can define the angles
$$ \varphi_j(u):={\rm arg}(1|e_j)^2\quad j\ge 1\ .$$
We do the same analysis with the operator $K_u=H_u T_z$. 
As before, by the antilinearity of $K_u$ there exists an orthonormal family $(f_j)_{j\ge 1}$ of the range of $K_u$ such that
$$K_u(f_j)=\sigma _jf_j\ ,\  j\ge 1,$$
and the family is determined by $u$ up to a change of sign on 
some of the $f_j$. One has also $(u\vert f_j)\neq 0$ because of the assumption on the $\rho_j$'s and $\sigma_j$'s. 
We set $$\theta_j(u):={\rm arg}(u|f_j)^2,\; j\ge 1\ .$$
Our main result is the following.
\begin{theo}\label{TheoHomeo}
The mapping   
$$\chi:=u\in VMO_{+,\rm{gen}}\mapsto \zeta=((\zeta_{2j-1}=\rho_je^{-i\varphi_j})_{j\ge 1}, (\zeta_{2j}=\sigma_je^{-i\theta_j})_{j\ge 1})$$  is a homeomorphism
onto
$$\Xi:=\{(\zeta_j)_{j\ge 1}\in \C^{\Z_+}, \; |\zeta_1|> |\zeta_2|>|\zeta_3|>|\zeta_4|\dots>\dots \to 0\}.$$
Moreover, one has an explicit formula for the inverse mapping. Namely, if $\zeta$ is given in $\Xi$, then the Fourier coefficients of $u$ are  given by
\begin{equation}\label{c_nDiffeo}
\hat u(n)=X.A^nY\ ,
\end{equation}
where $A=(A_{jk})_{j,k\ge 1}$ is the bounded operator on $\ell ^2$ defined by
\begin{equation}\label{A}
A_{jk}=\sum _{m=1}^\infty \frac{\nu _j\nu _k\zeta _{2k-1}\kappa _m^2\zeta _{2m}}{(\vert \zeta _{2j-1}\vert ^2-\vert \zeta _{2m}\vert ^2)(\vert \zeta _{2k-1}\vert ^2-\vert \zeta _{2m}\vert ^2)}\ ,\ j,k\ge 1\ ,
\end{equation}
with
\begin{equation}\label{nulambdamu}
\nu _j^2:=\left (1-\frac{\sigma _j^2}{\rho _j^2}\right )\prod _{k\ne j}\left (\frac{\rho _j^2-\sigma _k^2}{\rho _j^2-\rho _k^2}\right )\ ,
\end{equation}
 \begin{equation}\label{kappa}
\kappa_m^2:=\left (\rho_m^2-\sigma _m^2\right )\prod _{\ell\ne m}\left (\frac{\sigma_m^2-\rho _\ell^2}{\sigma_m^2-\sigma_\ell^2}\right )\ ,
\end{equation}
\begin{equation}\label{XY}
X=(\nu _j\zeta _{2j-1})_{j\ge 1}\ ,\  Y=(\nu _j)_{j\ge 1}\ ,\ 
\end{equation}
and $$V.W:=\sum _{j=1}^\infty v_jw_j\ {\rm if}\ V=(v_j)_{j\ge 1}, W=(w_j)_{j\ge 1}\ .$$
\end{theo}
Theorem \ref{TheoHomeo} calls for several comments. Firstly, it is not difficult to see that the first part of Theorem \ref{HankelautoE} is a direct consequence of Theorem \ref{TheoHomeo} (see the end of Section 3 below).  More generally, as an immediate corollary of Theorem \ref{TheoHomeo}, one shows that, for any given sequences $(\rho_j)_{j\ge 1}$ and $(\sigma_j)_{j\ge 1}$ satisfying
$$\rho_1> \sigma_1> \rho_2>\sigma_2> \dots \to 0,$$
there exists an infinite dimensional torus of symbols $c$ such that the $(\rho_j)_{j\ge 1}$'s are the non zero singular values of $\Gamma_c$, and the  $(\sigma_j)_{j\ge 1}$'s are the non zero singular values of $\Gamma_{\tilde c}$.

Next we make the connection with previous results. In a preceding article (\cite{GG2}),  we have obtained an analogue of Theorem \ref{TheoHomeo} in the more restricted context of 
Hilbert-Schmidt Hankel operators. This result arises in \cite{GG2} as a byproduct of the study of the dynamics of some completely integrable Hamiltonian system called the cubic Szeg\"o equation (see \cite{GG} and \cite{GG2}). In this setting the phase space of this Hamiltonian system is the Sobolev space $H^{1/2}_+$, which is the space of symbols of Hilbert-Schmidt Hankel operators, and the restriction of the mapping $\chi$ to the phase space can be interpreted as an action-angle map. In the present paper, we extend this result to compact Hankel operators, which is the natural setting for an inverse spectral problem.

Finally, we would like to comment about   the above explicit formula giving $\hat u(n)$. The boundedness of operator $A$ defined by (\ref{A}) is not trivial. In fact, it is a consequence of the proof of the theorem.
However, it is possible to give a direct proof of this boundedness, see Appendix 2. Furthermore, from the complicated structure of  formula (\ref{c_nDiffeo}), it seems difficult to check directly that the corresponding Hankel operators have the right sequences of singular values, namely that the map $\chi $ is onto. Our proof is in fact completely different and is based on some compactness argument, while, as in \cite{GG2}, the explicit formula is only used to establish the injectivity of $\chi $.

We now state our last result, which describes the kernel of $H_u$ in terms of the $\zeta =\chi (u)$. 

As $\ker H_u$ is invariant by the shift,  the Beurling theorem --- see {\it e.g.} \cite{R}--- provides the existence of an inner function $\varphi$ so that $\ker H_u=\varphi L^2_+$. We use the notation of Theorem \ref{TheoHomeo} to describe $\varphi$. Denote by $R$ the range of $H_u$.

\begin{theo}\label{KerHu}
We keep the notation of Theorem \ref{TheoHomeo}. Let $u\in VMO_{+,\rm gen}$. 
The kernel of $H_u$ and the kernel of $K_u$ are reduced to zero if and only if $1\in\overline R\setminus R$ or if and only if the following conditions hold.
\begin{equation}\label{HuOneToOne}
\sum_{j=1}^\infty \left( 1-\frac{\sigma_j^2}{\rho_j^2}\right)=\infty,\; \;\sup_N\frac 1{\rho_{N+1}^2}\prod_{j=1}^N\frac{\sigma_j^2}{\rho_j^2}=\infty.\end{equation}
When these conditions are not satisfied, $\ker H_u=\varphi L^2_+$  with $\varphi$ inner satisfying
\begin{enumerate}
\item if $1$ does not belong to the closure of the range of $H_u$ {\it i.e.} $1\notin\overline R$, then
$$\varphi(z)=(1-\sum \nu_j^2)^{-1/2}(1-\sum_{n\ge 0} \alpha_n z^n)$$
where 
\begin{equation}\label{phi}
\alpha_n=Y.A^nY\end{equation}
Furthermore, $\ker K_u=\ker H_u=\varphi L^2_+$.
\item if $1$ belongs to the range of $H_u$, {\it i.e.} $1\in R$, then  $\varphi(z)=z\psi(z)$  with
$$\psi(z)=\left(\sum_{j=1}^\infty\frac{\nu_j^2}{\rho_j^2}\right)^{-1/2}\sum_{n\ge 0} \beta_n z^n$$
where 
\begin{equation}\label{tildephi}
\beta_n=W.A^nY\ ,\ W=(\nu _j\zeta _{2j-1}\rho _j^{-2})_{j\ge 1}\ .\end{equation}
Furthermore, $\ker K_u=\ker H_u\oplus\C H_u^{-1}(1)=\varphi L^2_+\oplus \C \psi$.
\end{enumerate} 

\end{theo}

We end this introduction by describing the organization of this paper. In Section 2, we start the proof of Theorem \ref{TheoHomeo}. We first recall from \cite{GG2} a finite dimensional analogue to Theorem 
\ref{TheoHomeo}. Then we generalize from \cite{GG2} an important trace formula to arbitrary compact Hankel operators. We then use this formula and the Adamyan-Arov-Krein theorem to derive 
a crucial compactness lemma about Hankel operators. Using this compactness lemma, we prove Theorem \ref{TheoHomeo} in Section 3, and we infer the first part of Theorem \ref{HankelautoE}.
Section 4 is devoted to the proof of Theorem \ref{KerHu}, from which the second part of Theorem  \ref{HankelautoE} easily follows. Finally, for the convenience of the reader, we have gathered 
in Appendix 1 the main steps of the proof of the finite dimensional analogue of Theorem \ref{TheoHomeo}, while Appendix 2 is devoted to a direct proof of the boundedness of operator $A$
involved in Theorem \ref{TheoHomeo}.

\section{Preliminary results}

 The  proof of Theorem \ref{TheoHomeo}  is based on a  finite rank approximation of $H_u$. We first recall the notation and a similar  result obtained on finite rank operators in \cite{GG2}.

\subsection{The finite rank result}

By a theorem due to Kronecker (\cite{Kr}),
the Hankel operator $H_u$ is of finite rank if and only if $u$ is a rational function, holomorphic in the unit disc.
As in the introduction, we consider ${\mathcal V}(2N)$  the set of rational functions $u$, holomorphic in the unit disc, so that $H_u$  and $K_u$ are of finite rank $N$.  
It is elementary to check that ${\mathcal V}(2N)$ is a
$2N$-dimensional complex submanifold of $L^2_+$ (we refer to \cite{GG} for a complete description of this set and for an elementary proof of Kronecker Theorem). 
We denote by $\mathcal V(2N)_{\rm gen}$ the set of functions  $u\in {\mathcal V}(2N)$ such that $H_u^2$ and $K_u^2$ have simple distinct eigenvalues $(\rho_j^2)_{1\le j\le N}$ and $(\sigma_m^2)_{1\le m\le N}$ respectively with
$$\rho_1^2>\sigma_1^2>\rho_2^2>\dots\rho_N^2>\sigma_N^2>0.$$

As in the introduction, we can define new variables on $\mathcal V(2N)_{\rm gen}$ and a corresponding mapping $\chi_N$. 
The following result has been proven in \cite{GG2}.
\begin{theo}\label{TheoDiffeoMN}

The mapping   
$$\chi_N:=u\in\mathcal V(2N)_{\rm{gen}}\mapsto \zeta=(\zeta_{2j-1}=\rho_je^{-i\varphi_j}, \zeta_{2j}=\sigma_je^{-i\theta_j})_{1\le j \le N}$$  is a symplectic diffeomorphism
onto
$$\Xi_N:=\{\zeta\in\C^{2N}, \; |\zeta_1|>|\zeta_2|> |\zeta_3|>|\zeta_4|>\dots>|\zeta_{2N-1}|>|\zeta_{2N}|>0\}$$
in the sense that the image of the symplectic form $\omega$ by $\chi_N$ satisfies
\begin{equation}\label{ChiStarOmegaN}
({\chi_N})_{*}\omega=\frac 1{2i}\sum_{1\le j\le 2N}d\zeta_j\wedge d\overline{\zeta_j}.
\end{equation}
\end{theo}
There is also an explicit formula for the inverse $\chi_N$ analogous to the one given in Theorem  \ref{TheoHomeo} except that the sums in formulae (\ref{c_nDiffeo}) run over the integers $1,\dots, N$.

In order  to prove the extension of Theorem \ref{TheoDiffeoMN} to $VMO_{+,\rm{gen}}$, we have to extend some tools introduced in \cite{GG2}.

\subsection{The functional $J(x)$}

Let $H$ be a compact selfadjoint antilinear  operator on a Hilbert space $\mathcal H$. Let $A=H^2$ and $e\in\mathcal H$ so that $\Vert e\Vert =1$. Notice that $A$ is selfadjoint, positive and compact. We define the generating function of $H$ for $|x|$ small, by  $$J(x)(A)=1+\sum_{n=1}^\infty x^nJ_{n}$$ where $J_{n}=J_{n}(A)=(A^{n}(e)\vert e)$. Consider the operator $$B:=A-(\;\cdot\;\vert H(e))H(e)$$ which is also selfadjoint, positive and compact. Denote by $(a_j)_{j\ge 1}$ (resp. $(b_j)_{j\ge 1}$) the non-zero eigenvalues of $A$ (resp. of $B$) labelled according to the min-max principle, 
$$a_1\ge b_1\ge a_2\ge  \dots $$
Notice that
$$J(x)(A)=((I-xA)^{-1}(e)\vert e)\ $$
which shows that $J$ extends as an entire meromorphic function, with poles at $x=\frac 1{a_j}, j\ge 1$.

 \begin{proposition}\label{trace}
  \begin{equation}\label{ProdHmu}
 J(x)(A)=\prod_{j=1}^\infty \frac{1-b_jx}{1-a_jx}\ ,\  x\notin \left \{ \frac 1{a _j}, j\ge 1\right \}.
 \end{equation}
\end{proposition}

\begin{proof} We first assume $A$ and $B$ in the trace class.  In that case, we can compute the trace of $(I-xA)^{-1}-(I-xB)^{-1}$. We first write
$$[(I-xA)^{-1}-(I-xB)^{-1}](f)=\frac x{J(x)}(f\vert (I-xA)^{-1}H(e))\cdot (I-xA)^{-1} H(e).$$
Consequently, taking the trace, we get
$${\rm Tr}[(I-xA)^{-1}-(I-xB)^{-1}]=\frac x{J(x)}\Vert (I-xA)^{-1} H(e)\Vert^2.$$
As, on the one  hand,
$$\Vert (I-xA)^{-1} H(e)\Vert^2= ((I-xA)^{-1}A(e)\vert e)=J'(x)$$
and on the other hand
\begin{eqnarray*}
{\rm Tr}[(I-xA)^{-1}-(I-xB)^{-1}]&=&x{\rm Tr}[A(I-xA)^{-1}-B(I-xB)^{-1}]\\
&=&x \sum _{j=1}^\infty\left ( \frac{a_j}{1-a_jx}-\frac{b_j}{1-b_jx}\right )\end{eqnarray*}
 we get
 \begin{equation}\label{traceformula}
  \sum _{j=1}^\infty\left ( \frac{a_j}{1-a_jx}-\frac{b_j}{1-b_jx}\right )=\frac{J'(x)}{J(x)}\ ,\ x\notin \left \{ \frac 1{a_j}, \frac 1{b_j}, j\ge 1\right \}\ .
 \end{equation}
 From this equation, one gets easily formula (\ref{ProdHmu}) for $A$ and $B$ in the trace class. To extend it to compact operators, we first recall that  $$a_j\ge b_j\ge a_{j+1}.$$ Hence, $\sum_j(a_j-b_j)$ converges  when $A$ is compact since $0\le a_j-b_j\le a_j-a_{j+1}$ and $a_j$ tends to zero by compactness of $A$.
Hence, the infinite product in Formula (\ref{ProdHmu}) converges, and the above computation makes  sense for compact operators.
\end{proof}
 
 \begin{Lemma}\label{EigenvalueLimit}
 Let  $e\in\mathcal H$ with $\Vert e\Vert=1$.
Let $(H_p)$ be a sequence of compact selfadjoint antilinear  operators on a Hilbert space $\mathcal H$ which converges strongly to $H$, namely
$$\forall h\in \mathcal H\ ,\ H_ph\td_p,\infty Hh\ .$$
We assume that $H$ is compact.
Let $A_p=H_p^2$, $B_p=A_p-(\;\cdot\;\vert H_p(e))H_p(e)$,  and  $A=H^2$ et $B=A-(\;\cdot\;\vert H(e))H(e)$ their strong limits.  
For every $j\ge 1$, denote by $\mathcal F _j$ the set of linear subspaces of $\mathcal H$ of dimension at most $j$,  set
$$a_j^{(p)}=\min _{F\in \mathcal F_{j-1}}\max _{h\in F^\perp, \Vert h\Vert =1}(A_p(h)\vert h)\ ,$$
$$b_j^{(p)}=\min _{F\in \mathcal F_{j-1}}\max _{h\in F^\perp, \Vert h\Vert =1}(B_p(h)\vert h)\ .$$

Assume there exist $(\overline a_j)$ and $(\overline b_j)$ such that
 $$\sup_{j\ge 1}|a_j^{(p)}-\overline a _j|\td_p,\infty  0\ ,\ \sup_{j\ge 1}|b_j^{(p)}-\overline b _j|\td_p,\infty 0,$$
and  the non-zero $\overline a_j$, $\overline b_m$ are pairwise distinct. Then the positive eigenvalues of $A$ are simple and are exactly the $\overline a_j$'s; similarly, the positive eigenvalues of $B=A-(\;\cdot\;\vert H(e))H(e)$ are simple and are exactly the $\overline b_m$'s.
\end{Lemma}  
\begin{proof}
By assumption, for every $h\in \mathcal H$, we have 
\begin{equation}\label{strongHu}
A_p(h)\td _p,\infty A(h)\ .
\end{equation}
Since the norm of $A_p$ is uniformly bounded,
we conclude that (\ref{strongHu}) holds uniformly for $h$ in every compact subset of $\mathcal H$, hence 
$$\forall n\ge 1, A_p^n(h)\td _p,\infty A^n(h)\ .$$

In particular, for every $n\ge 1$,
$$J_{n}(A_p):=(A_p^n(e)\vert e)\td _p,\infty  (A^n(e)\vert e):=J_{n}(A)\ ,$$
and there exists $C>0$ such that
$$\forall n\ge 1, \;\sup _p J_{n}(A_p)\le C^n\ .$$
Choose $\delta >0$ such that $\delta C<1$. Then, for every real number $x$ such that $\vert x \vert <\delta $, we have, by dominated convergence,
$$J(x)(A_p):=1+\sum _{n=1}^\infty x^nJ_{n}(A_p)\td _p,\infty 1+\sum _{n=1}^\infty x^nJ_{n}(A):=J(x)(A)\ .$$
On the other hand, in view of the assumption about the convergence of $(a _j^{(p)})_{j\ge 1}$ and $(b _j^{(p)})_{j\ge 1}$ and the convergence of the product in Formula (\ref{ProdHmu}), we also have, for $\vert x\vert < \delta $,
\begin{equation}\label{idlimit}J(x)(A_p)=\prod _{j=1}^\infty\left ( \frac{1-b _j^{(p)}x}{1-a_j^{(p)}x}\right )\td _p,\infty   \prod _{j=1}^\infty\left ( \frac{1-\overline{b _j}x}{1-\overline{a_j}x}\right ).\ \end{equation}
Hence, we obtain 
\begin{equation}
J(x)(A)= \prod _{j=1}^\infty\left ( \frac{1-\overline{b _j}x}{1-\overline{a_j}x}\right ).\end{equation}

By assumption, the non-zero $\overline a_j$, $\overline b_m$ are pairwise distinct so no cancellation can occur in the right hand side of (\ref{ProdHmu}), and the poles are all distinct. 

On the other hand, denote by $(a_j)$  the family of eigenvalues of $A$ and by $(b_j)$ the one of $B$. 
By a classical result (see e.g. Lemma 1, section 2.2 of \cite{GG2}), 
$$\{ a_j, j\ge 1\} \subset \{ \overline a _j, j\ge 1\} \ ,\ \{ b _j, j\ge 1\} \subset \{ \overline b_j, j\ge 1\}$$
and the multiplicity of positive eigenvalues is $1$. Consequently, there is no cancellation in the expression of $J(x)(A)$ and all the poles are simple.
We conclude that $a_j=\overline a_j$, $b_j=\overline b _j$ for every $j\ge 1$.
\end{proof}
\subsection{A compactness result}
 From now on, we choose $\mathcal H=L^2_+$ and $e=1$.
As a first application of Proposition \ref{trace}, we obtain the following.
\begin{Lemma}\label{Nu_j}
For any $u\in VMO_+(\T)$, we have
$$
J(x):=J(x)(H_u^2)
=\prod_{j=1}^\infty\frac{1-x\sigma_j^2(u)}{1-x\rho_j^2(u) }=1+x\sum_{j=1}^\infty\frac{\rho_j^2(u)\nu_j^2}{1-x\rho_j^2(u)},\; x\notin\left\{\frac{1}{\rho_j^2(u)}\right\}_{j\ge 1}.
$$
 Here $\nu_j:=|(1\vert e_j)|.$ In particular, 
$$\nu_j^2=\left(1-\frac{\sigma_j^2}{\rho_j^2}\right)\prod_{k\neq j}\left(\frac{\rho_j^2-\sigma_k^2}{\rho_j^2-\rho_k^2}\right)$$
\end{Lemma}
 
 The first equality is just a consequence of (\ref{ProdHmu}). For the second equality, we use the formula $J(x)=((I-xH_u^2)^{-1}(1)\vert 1)$ and we expand 
 $1$ according to the decomposition
 $$L^2_+=\oplus _{j\ge 1} \C e_j \oplus \ker H_u .$$

From Lemma \ref{EigenvalueLimit}, we infer the following compactness result, which can be interpreted as a compensated compactness result.
 \begin{proposition}\label{compactness}
 Let $(u_p)$ be a sequence of $VMO_+(\T)$ weakly convergent to $u$ in $VMO_+(\T)$. We assume that,  for some sequences
  $(\overline \rho_j)$ and $(\overline\sigma_j)$, 
 $$\sup_{j\ge 1}|\rho _j(u_p)-\overline \rho _j|\td_p,\infty  0\ ,\ \sup_{j\ge 1}|\sigma _j(u_p)-\overline \sigma _j|\td_p,\infty 0,$$
 and the following simplicity assumption:
all the non-zero $\overline \rho_j$, $\overline\sigma_m$ are pairwise distinct.
 Then, for every $j\ge 1$, $\rho _j(u)=\overline\rho _j$, $\sigma _j(u)=\overline\sigma _j$, and the convergence of $u_p$ to $u$ is strong in $VMO_+(\T)$.
 \end{proposition}
 \begin{remark}Let us emphasize that this result specifically uses the structure of Hankel operators. It is false in general for compact operators assumed to converge only strongly. One  also has to remark that the simplicity of the eigenvalues is a crucial hypothesis 
 as the following example shows.
 Denote by $(u_p)$, $|p|<1 $, $p$ real, the sequence of functions defined by
 $$u_p(z)=\frac {z-p}{1-pz}.$$
 Then, the selfadjoint Hankel operators $H_{u_p}$ and $K_{u_p}$ have eigenvalues $\lambda_1=\mu_1=1$ and $\lambda_2=-1$ and $\mu_m=\lambda_{m+1}=0$ for $m\ge 2$ independently of $p$ . As $p$ goes to $1$, $p<1$, $u_p$ tends weakly to the constant function $-1$, hence the convergence is not strong in $VMO$. Indeed, $H_{-1}$ is the rank one operator given by $H_{-1}(h)=-(1\vert h)$ hence $H^2_{-1}$ is a rank one projector while $H_{u_p}^2$ is a rank two projector. Therefore $$\Vert H_{u_p}^2-H_{-1}^2\Vert \ge 1\  {\rm since}\  {\rm Ran}H^2_{u_p}\cap \ker H_{-1}^2\neq \{0\}\ .$$
 \end{remark} 
 
\begin{proof}
Let us first recall the Adamyan-Arov-Krein  (AAK) Theorem  on approximation of Hankel operators by finite rank Hankel operators.
\begin{theo}[Adamyan-Arov-Krein \cite{AAK}] Let $\Gamma$ be a bounded Hankel operator on $L^2_+(\T)$. Let $(s_m(\Gamma))_{m\ge 1}$ be the family of singular values of $\Gamma$ labelled according to the min-max principle.Then, for any $m\ge 1$, there exists a Hankel operator $\Gamma_m$ of rank $m-1$ such that
$$s_m(\Gamma)=\Vert \Gamma-\Gamma_m\Vert.$$
\end{theo}
In other words, AAK Theorem states that the $m$-th singular value of a Hankel operator, as the distance of this operator to operators of rank $m-1$,  is exactly achieved by some Hankel operator of rank $m-1$, hence, with a rational symbol.

This result is crucial to obtain our compactness result.
We want to apply Lemma \ref{EigenvalueLimit} with $A=H_u^2$ and $B=K_u^2$ and $e=1$. One has to prove that, for any $h\in L^2_+$,
$H_{u_p}^2(h)\to H_u^2(h)$. By AAK Theorem, for any $p$ and any $j\ge 1$, there exists a function $u_{p,j}\in \mathcal V(2j)\cup {\mathcal V}(2j-1)$ so that 
$$\Vert H_{u_p}-H_{u_{p,j}}\Vert =\rho_{j+1}({u_p}).$$

In particular, we get
$$\Vert u_p-u_{p,j}\Vert_{L^2}\le \rho_{j+1}(u_p).$$
On the other hand, one has $$\Vert H_{u_{p,j}}\Vert \ge \frac{1}{\sqrt j}(Tr(H^2_{u_{p,j}}))^{1/2}\ge \frac 1{\sqrt j}\Vert u_{p,j}\Vert_{H^{1/2}}.$$ Hence, for fixed $j$, the sequence $(u_{p,j})_p$ is bounded in $H^{1/2}$.
We are going to prove that the sequence $\{u_p\}_p$ is precompact in $L^2_+$.
We show that, for any $\varepsilon>0$ there exists a finite sequence $v_k\in L^2_+$, $1\le k\le M$ so that $\{u_p\}_p\subset\cup_{k=1}^M B_{L^2_+}(v_k,\varepsilon)$.
Let $j$ be fixed so that $\sup _p\rho_{j+1}(u_p)\le \varepsilon/2$. As the sequence $(u_{p,j})_p$ is uniformly bounded in $H^{1/2}$, there is a subsequence which converges weakly in $H^{1/2}$. In particular, it is precompact in $L^2_+$ hence there exists $v_k\in L^2_+$, $1\le k\le M$ so that $\{u_{p,j}\}_p\subset\cup B_{L^2_+}(v_k,\varepsilon/2)$. Then, for every $p$ there exists a $k$ such that
$$\Vert u_p-v_k\Vert_{L^2}\le \rho_{j+1}(u_p)+\Vert u_{p,j}-v_k\Vert_{L^2}\le \varepsilon.$$
Therefore $\{ u_p \}$ is precompact in $L^2_+$ and, since $L^2$ is complete, some subsequence of $(u_p)$  has a strong limit in $L^2_+$. Since $u_p$ converges weakly to $u$, this limit has to be $u$, and
we conclude that the whole sequence $(u_p)$ is strongly convergent  to $u$ in $L^2_+$. Since $\Vert H_{u_p}\Vert \simeq \Vert u_p\Vert _{BMO}$ is bounded, we infer the strong convergence of  operators,
$$\forall h\in L^2_+, H_{u_p}(h)\td_p,\infty H_u(h)\ .$$
By Lemma \ref{EigenvalueLimit}, for every $k$ we have $\rho _k(u)=\overline \rho _k$ and $\sigma _k(u)=\overline \sigma _k$. We now want to prove that
$$\Vert H_{u_p}-H_u\Vert \rightarrow 0\ .$$
Let us distinguish two cases.
\s
{\it First case} : for every $j\ge 1$, $\overline \rho_j>0\ .$ We come back to the AAK situation above. For every $j$, we select $u_{p,j}\in \mathcal V(2j)\cup {\mathcal V}(2j-1)$ so that
$$\Vert H_{u_p}-H_{u_{p,j}}\Vert =\rho_{j+1}({u_p}).$$
Since the operator norm is lower semicontinuous for the strong convergence, we infer that any limit point $\tilde u_j$ of $u_{p,j}$  in $L^2_+$ as $p\rightarrow \infty $, satisfies
$$\Vert H_u-H_{\tilde u_j}\Vert \leq \overline \rho _{j+1}\ .$$
In particular, $\vert \overline \sigma _j-\sigma _j(\tilde u_j)\vert \le \overline \rho _{j+1}$, hence $\sigma _j(\tilde u_j)>0$ and thus $\tilde u_j\in \mathcal V(2j)$.
Using the following elementary lemma, the proof is then completed by the triangle inequality.
\begin{Lemma}\label{V(D)}
Let $N$ be a positive integer and  $w_p\in \mathcal V(2N)\cup \mathcal V(2N-1)$ such that $w_p\td_p,\infty w$ in $L^2_+$. Assume $w\in \mathcal V(2N)\cup \mathcal V(2N-1)$. Then $\Vert H_{w_p}-H_w\Vert \td_p,\infty 0.$
\end{Lemma}
Let us postpone the proof of Lemma \ref{V(D)} to the end of the argument.\\
\noindent {\it Second case} : there exists $k\ge 1$ such that $\overline \rho _k$=0. We denote by $j$ the greatest $k\ge 1$ such that $\overline \rho _k>0$. Of course we may assume that there exists  such a $j$, otherwise this would mean that $\Vert H_{u_p}\Vert $ tends to $0$, a trivial case. For such a $j$, we  again write
$$\Vert H_{u_p}-H_{u_{p,j}}\Vert =\rho_{j+1}({u_p}),$$
and, passing to the limit, we conclude that $u_{p,j}$ is strongly convergent to $u$ in $L^2_+$. Using again Lemma \ref{V(D)}, we conclude that $\Vert H_{u_{p,j}} -H_u\Vert $ tends to $0$, and the proof is again completed by the triangle inequality.
\s
Finally, let us prove Lemma \ref{V(D)}. Recall the explicit description of $\mathcal V(D)$, see e.g. \cite{GG}. Elements of $\mathcal V(2N)$ are rational functions of the following form,
$$w(z)=\frac{A(z)}{B(z)}\ ,$$
where  $A$, $B$ have no common factors, $B$ has no zeroes in the closed unit disc, $B(0)=1$, and ${\rm deg}(A)\leq N-1$, ${\rm deg}(B)=N$. Elements of $\mathcal V(2N-1)$ have the same form,
except that the last part is replaced by ${\rm deg}(A)= N-1$, ${\rm deg}(B)\le N-1$.
\s
Write similarly
$$w_p(z)=\frac{A_p(z)}{B_p(z)}\ .$$
By the Cauchy formula, we have, for every $z$ in the unit disc,
$$w_p(z)\rightarrow w(z)\ .$$
Since
$$B_p(z)=\prod _{k=1}^N(1-b_{k,p}z)$$
with $\vert b_{k,p}\vert <1$,we may assume that, up to extracting a subsequence,
$$B_p(z)\rightarrow \tilde B(z)=\prod _{k=1}^N(1-\tilde b_kz)\ ,$$
with $\vert \tilde b_k\vert \le 1$. Multiplying by $B_p(z)$ and passing to the limit, we get
$$A_p(z)\rightarrow \tilde B(z)\frac{A(z)}{B(z)}:= \tilde A(z)\ .$$
 Since $\tilde B A$ is divisible by $B$, $\tilde B$ is divisible by $B$. On the other hand, we claim that ${\rm deg}(\tilde B)\le {\rm deg}(B)$. Indeed, either $w\in \mathcal V(2N)$, and
 ${\rm deg}(B)=N\ge {\rm deg}(\tilde B)$ ; or $w\in \mathcal V(2N-1)$, and ${\rm deg}(A)=N-1\ge {\rm deg}(\tilde A)$. In both cases, we conclude $\tilde B=B$, which means that
 the numbers $b_{k,p}$ stay  away from the unit circle. Consequently, the convergence of $w_p(z)$ to $w(z)$ holds uniformly on a disc $D(0,r)$ for some $r>1$, thus, say, $w_p\rightarrow w$
 in $H^s(\T )$ for every $s>0$. Choosing $s=\frac 12$, we conclude that $H_{w_p}$ converges to $H_w$ in the Hilbert-Schmidt norm, hence in the operator norm.
\end{proof}

\section{Proof of Theorem \ref{TheoHomeo} and of the first part of Theorem \ref{HankelautoE}}

\subsection{The surjectivity of $\chi $.} 
Let $(\zeta_p)_{p\ge 1}$ be an element in $\Xi$. We want to prove the existence of $u\in VMO_{+,{\rm gen}}$ so that $\chi (u)=(\zeta_{p})_{p\ge 1}$. We are going to use the finite rank result.
By Theorem \ref{TheoDiffeoMN}, for every $N$ we construct $u_N\in\mathcal V(2N)$ via the diffeomorphism $\chi_N$ by letting
$$\chi _N(u_N)=(\zeta _p)_{1\le p\le 2N}\ .$$
The sequence $(u_N)$ satisfies $\Vert H_{u_N}\Vert =\rho _1(u_N)=\vert \zeta _1\vert $, hence is bounded in $VMO$, and therefore has a subsequence, still denoted by $(u_N)$,  which is weakly convergent to $u$ in $VMO_+$. 
We can then apply  Proposition \ref{compactness}, hence $u$ is  the strong limit of $(u_N)$ in $VMO_+(\T )$, so that
$$\rho _j(u)=\vert \zeta _{2j-1}\vert :=\rho _j\ ,\ \sigma _j(u)=\vert \zeta _{2j}\vert :=\sigma _j .$$
In particular, $u\in VMO_{+,{\rm gen}}$. It remains to consider the convergence of the angles and hence of the eigenvectors. Let $j$ be fixed. For $N>j$, denote by $e_{j,N}$ the normalized eigenvector of $H_{u_{N}}^2$ related to the simple eigenvalue $\rho_j^2$ so that $H_{u_{N}}(e_{j,N})=\rho  _j\vert e_{j,N}$. As $(e_{j,N})$ is a sequence of unitary vectors, it has a weakly convergent subsequence to some vector $\tilde e_j$. We now show that the convergence is in fact  strong.
 Let us consider the operator 
$$P_{j,N}=\int_{\mathcal C_j}(zI-H_{u_N}^2)^{-1} \frac{dz}{2i\pi}$$ where $\mathcal C_j$ is a small circle around $\rho_j^2$. If $\mathcal C_j$ is  sufficiently small 
$$P_{j,N}(h)=(h\vert e_{j,N})e_{j,N}.$$ By the convergence of $H_{u_N}$ to $H_u$, we have for any $h\in L^2$,
$$P_{j,N}(h)\to P_j(h)$$ where $P_j$ is the projector onto the eigenspace of $H_u^2$ corresponding to $\rho_j^2$. Denoting by $e_j$ a unitary vector of this eigenline, we get that, for any $h\in L^2$,
$$(h\vert e_{j,N})e_{j,N}\to (h\vert e_{j})e_j.$$

As $(h\vert e_{j,N})$ converges to $(h\vert \tilde e_j)$ by weak convergence, and on the other hand $\vert P_{j,N}(e_{j,N})\vert =|(h\vert e_{j,N})|$ tends to $\Vert P_j(h)\Vert =|(h\vert e_{j})|$, we get that $|(h\vert \tilde e_j)|=|(h\vert e_{j})|$ for any $h$ in $L^2$, hence $\tilde e_j=e^{i\Psi}e_j$ is unitary. We conclude that the convergence of $e_{j,N}$ to $\tilde e_j$ is strong since the convergence is weak and the vectors are unitary.
Hence $H_{u_N}(e_{j,N})=\rho_{j,N} e_{j,N}$ converges to $H_u(\tilde e_j)=\rho_j\tilde e_j$, and the angles ${\rm arg}(1\vert e_{j,N})^2$ converge to ${\rm arg} (1\vert \tilde e_j)^2$. The same holds for the eigenvectors of $K_{u_N}$. We conclude that there exists $u\in VMO_{+,\rm gen}$ with $\chi(u)=(\zeta_p)_{p\ge 1}$. The mapping $\chi$ is onto. 

The second step is  to prove that $\chi $ is one-to-one. It comes from an explicit formula giving $u$ in terms of $\chi(u)$.
 
\s

\subsection{An explicit formula via the compressed shift operator}
We are going to use the well known link between the shift operator and the Hankel operators. Namely, if $T_z$ denotes the shift operator, one can easily check the following identity,
\begin{equation}\label{Shift}H_uT_z=T_z^*H_u.\end{equation}
With the notation introduced in the introduction, it reads
$$K_u=T_z^*H_u.$$
Moreover,
$$K_u^2=H_uT_zT_z^*H_u=H_u(I-(\, .\,  \vert 1))H_u=H_u^2-(\, .\, \vert u)u\ .$$
We introduce the compressed shift operator  (\cite{N}, \cite{N2}, \cite{P})
$$S:=P_uT_z\ ,$$
 where $P_u$ denotes the orthogonal projector onto the closure of the range of $H_u$. 
By property (\ref{Shift}), $\ker H_u=\ker P_u$ is stable by $T_z$, hence 
 $$S=P_uT_zP_u$$ so that $S$ is an operator from the closure of the range of $H_u$ into itself. In the sequel, we shall always denote by $S$ the induced operator
 on the closure of the range of $H_u$, and  by $S^*$ the adjoint of this operator. 
 \s
Now observe that  operator $S$ arises in the Fourier series decomposition of $u$, namely
$$
u(z)=\sum _{n=0}^\infty \hat u(n)z^n,
$$
where
\begin{equation}\label{FourierU}
\hat u(n)=(u|z^n)=(u|T^n_z(1))=(u|S^nP_u(1))\ .
\end{equation}
As a consequence, we have, for $|z|<1$,
\begin{equation}\label{InverseSpectral}
u(z)=(u|(I-\overline zS)^{-1}P_u(1)).
\end{equation}
which makes sense since $\Vert S\Vert \le 1$. 
By studying the spectral properties of $K_u^2$, one obtains the following lemma.
\begin{Lemma}\label{spectralS}
The sequence $(g_j)_{j\ge 1}$ defined by $g_j=(H_u^2-\sigma_jI)^{-1}(u)$ is an orthogonal basis of the range of $K_u$,
on which the compressed shift operator acts as
$$S(g_j)={\sigma _j}\, {\rm e}^{i\theta _j}\, h_j\ ,\ h_j:=(H_u^2-\sigma^2 _jI)^{-1}P_u(1)\ .$$
\end{Lemma}
To obtain an explicit formula from Formula (\ref{InverseSpectral}), it is sufficient to express the action of $S$ on a basis of the closure of the range of $H_u$.

Hence, when the closure of the range of $H_u$ and the closure of the range of $K_u$ coincide, one can conclude from this Lemma, Lemma \ref{Nu_j} and Equation (\ref{InverseSpectral}) and obtain the explicit formula writing everything in the basis $(\tilde e_j)_{j\ge 1}$ of $\overline R$, where 
\begin{equation}\label{basetilde}
\tilde e_j:={\rm e}^{i\varphi _j/2}\, e_j\ .
\end{equation}
If the range of $K_u$ is strictly included in the range of $H_u$, there exists $g$ in the range of $H_u$ so that $K_ug=0=T^*_zH_ug$ hence $H_ug$ is a non-zero constant, in particular $1$ belongs to the range of $H_u$. Let us write $1=H_ug_0$. In this case, an orthogonal basis of the closure of the range of $H_u$ is given by the sequence $(g_m)_{m\ge 0}$ and, as $K_u(g_0)=0=H_uS(g_0)$, $S(g_0)=0$. 
So we obtain the same explicit formula for $u$ in terms of $\chi(u)$. This proves that the mapping $\chi$ is one-to-one.
\s
To prove that $\chi$ is a homeomorphism, it remains to prove that $\chi^{-1}$ is continuous on $\Xi$. One has to prove that if $\chi(u_p)$ tends to $\chi(u)$ then $(u_p)$ tends to $u$ in $VMO$. It is straightforward from Proposition \ref{compactness} that $(u_p)$ has a subsequence which converges strongly to $v$ in $VMO$. As $\chi$ is continuous and one-to-one, we get $v=u$. 

\subsection{The case of real Fourier coefficients}
Finally, let us  infer the first part of Theorem \ref{HankelautoE} from  Theorem \ref{TheoHomeo}. Firstly, we claim that the elements of $VMO_{+,{\rm gen}}$ with real Fourier coefficients correspond via the map $\chi $ to 
elements $\zeta \in \Xi $ which are real valued. Indeed, if $\zeta $ is real valued, the explicit formula (\ref{c_nDiffeo}) clearly implies that $\hat u(n)$ is real for every $n$. Conversely, if $u\in VMO_{+,{\rm gen}}$
has real Fourier coefficients, then $H_u$ and $K_u$ are compact selfadjoint operators on the closed real subspace of $L^2_+$ consisting of functions with real Fourier coefficients. Consequently, they admit orthonormal bases of eigenvectors in this space. Therefore we can write 
$$H_u(\tilde e_j)=\lambda _j\tilde e_j\ ,\ \lambda _j=\pm \rho _j\ ,\ K_u(\tilde f_m)=\mu _m\tilde f_m\ ,\mu _m=\pm \sigma _m\ ,$$
where $\tilde e_j$ and $\tilde f_m$ are unitary vectors with real Fourier coefficients. Since $\rho _j^2$ and $\sigma _m^2$ are simple eigenvalues of $H_u^2$ and $K_u^2$ respectively, we conclude that
$\tilde e_j$ is collinear to $e_j$, and similarly $\tilde f_m$ is collinear to $f_m$. More precisely, since $H_u$ and $K_u$ are antilinear,
\begin{eqnarray*}
\tilde e_j= \begin{cases}  \pm e_j\ {\rm if}\ \lambda _j=\rho _j\\
\pm ie_j\ {\rm if}\ \lambda _j=-\rho _j\ 
\end{cases}\ \ ;\ \ \tilde f_m=\begin{cases}  \pm f_m\ {\rm if}\ \mu _m=\sigma  _m\\
\pm if_m\ {\rm if}\ \mu _m=-\sigma _m\ 
\end{cases}\ 
\end{eqnarray*}
Since $(1\vert \tilde e_j)$ and $(u\vert \tilde f_m)$ are real, we conclude that
\begin{eqnarray*}
\varphi _j= \begin{cases}  0\ {\rm if}\ \lambda _j=\rho _j\\
\pi \ {\rm if}\ \lambda _j=-\rho _j\ 
\end{cases}\ \ ;\ \ \theta_m=\begin{cases}  0\ {\rm if}\ \mu _m=\sigma  _m\\
\pi \ {\rm if}\ \mu _m=-\sigma _m\ 
\end{cases}\ 
\end{eqnarray*}
Therefore, $\zeta _{2j-1}=\lambda _j$ and $\zeta _{2j}=\mu _j$. This completes the proof.

\section {Proof of Theorem \ref{KerHu}}
\begin{proof}
We already observed that $\ker H_u\subset \ker K_u$ and that the inclusion is strict if and only if $1\in R$ and in that case, $\ker K_u=\ker H_u\cup \C H_u^{-1}(1)$. Hence, in the following, we focus on the kernel of $H_u$.

We first prove that $\ker H_u=\{0\}$ if and only if $1\in \overline R\setminus R$.

As $\ker H_u=\{0\}$ is equivalent to $\overline R=L^2_+$, $\ker H_u=\{0\}$ implies $1\in\overline R$. If $1\in R$, then there exists $w\in \overline R$ so that $1=H_u(w)$. If we introduce the function $\psi=zw$, then $H_u(\psi)=T_z^*H_u(w)=T_z^*(1)=0$. It implies  that $\psi$ belongs to $\ker H_u$ and $\psi\neq 0$. Hence, $\ker H_u=\{0\}$ implies $1\in\overline R\setminus R$.

Let us prove the converse. Assume that $\ker H_u\neq\{0\}$ and that $1\in\overline R$. Let us show that $1\in R$. By the Beurling Theorem, we have $\ker H_u=\varphi L^2_+$ for some inner function $\varphi $. As   $1$ belongs to  $\overline R$, it is orthogonal to $\ker H_u$ hence $(1\vert\varphi)=0$. It implies that $\varphi=zw$ for some $w$ and, as $H_u(\varphi)=0=T_z^*H_u(w)$, we get that $H_u(w)$ is a non zero constant (if $H_u(w)=0$, $w$ should be divisible by $\varphi$ which is impossible since $\varphi=zw$). Eventually, we get that the constants are in $R$ and so is $1$. Hence we proved that  $\ker H_u\neq\{0\}$ if and only if either $1$ belongs to $R$ or $1$ does not belong to $\overline R$.

It remains to prove that the property $1\in\overline R\setminus R$ is equivalent to equations (\ref{HuOneToOne}) that we recall here,
$$
\sum_{j=1}^\infty \left( 1-\frac{\sigma_j^2}{\rho_j^2}\right)=\infty,\; \;\sup_N\frac 1{\rho_{N+1}^2}\prod_{j=1}^N\frac{\sigma_j^2}{\rho_j^2}=\infty.$$
Firstly, $1\in\overline R$ if and only if $\sum_{j=1}^\infty \nu_j^2=1$ which, in turn, letting $x$ tend to $\infty$ in formula giving $J(x)$ in Lemma \ref{Nu_j}, is equivalent to $$\prod_{j=1}^\infty \frac{\sigma_j^2}{\rho_j^2}=0.$$ It gives the first condition. We claim that $1$ belongs to $R$ if and only if  $$\sum_{j=1}^\infty \frac{\nu_j^2}{\rho_j^2}<\infty.$$ Indeed, it is a necessary and sufficient condition to be able to define $$w=\sum_{j=1}^\infty \frac{\nu_j}{\rho_j}e^{-i\varphi_j/2} e_j$$ so that $H_u(w)=1$. We now show that this condition is equivalent to $$\sup_{N}\frac 1{\rho_{N+1}^2}\prod_{j=1}^N\frac{\sigma_j^2}{\rho_j^2}<\infty.$$ Let us denote by $p_N$ the quantity $$p_N:=\frac 1{\rho_{N+1}^2}\prod_{j=1}^N\frac{\sigma_j^2}{\rho_j^2}$$ and let us show that $\sup_N p_N<\infty$.
Indeed, the sequence $(p_N)$ is increasing and \begin{equation}\label{limF(y)}\sum_{j=1}^\infty \frac{\nu_j^2}{\rho_j^2}=-\lim_{x\to\infty}xJ(x)=\lim_{y\to\infty} F(y)\ ,\ F(y):=y\prod_{j=1}^\infty\frac{1+y\sigma_j^2}{1+y\rho_j^2}\ .\end{equation} (here we used Lemma \ref{Nu_j} and the equality $\sum_{j=1}^\infty\nu_j^2=1$ so that $J(x)=\sum_{j=1}^\infty\frac{\nu_j^2}{1-x\rho_j^2(u)}$). 
Let us define $$F_N(y)=\frac y{1+y\rho _{N+1}^2}\prod _{j=1}^N\frac{1+y\sigma _j^2}{1+y\rho _j^2}.$$Then, this quantity 
is increasing with respect to $N$ and to $y$ hence $$\sup_N p_N=\sup_N\sup_y F_N(y)=\sup_y\sup_N F_N(y)=\sup_y F(y)<\infty.$$

Now, we prove the formulae (\ref{phi}) and (\ref{tildephi}) which give the generators of the kernels. We first consider the case when $1\notin \overline R$. As $1-P_u(1)$ belongs to $\ker H_u$, $1-P_u(1)=\varphi f$ for some $f\in L^2_+$. Let us remark that for any $h\in \ker H_u$, $\overline {1-P_u(1)} h$ is holomorphic. Indeed, for any $k\ge 1$, one has
$$(\overline {1-P_u(1)}h\vert \overline z^k)=(z^kh\vert 1-P_u(1))=0-(z^kh\vert P_u(1))=0,$$ 
the last equality coming from the fact that $z^kh\in\ker H_u$.
Since the modulus of $\varphi $ is $1$, it implies that $\overline f$ is holomorphic  hence it is a constant.  We get that $\varphi=\frac{1-P_u(1)}{\Vert 1-P_u(1)\Vert}$.
One can write, as for formula (\ref{FourierU}), $$1-P_u(1)=1-\sum (P_u(1)\vert S^n P_u(1))z^n$$ and the explicit formula is obtained by writing this equality in the orthogonal basis $(\tilde e_j)$ defined by (\ref{basetilde}).

It remains to consider the case $1\in R$. Then, one can choose $w\in \overline R$ so that $H_u(w)=1$. In particular, $H_u(zw)=T^*_zH_u(w)=0$ so that $zw=\varphi f$ for some $f$ in $L^2_+$. As before, one can prove that, for any $h\in\ker H_u$, $\overline {zw}h$ is holomorphic hence $\overline f$ is holomorphic hence is constant. Eventually, in this case, we obtain $\varphi=z\frac w{\Vert w\Vert}$. The explicit formula follows from direct computation as before.

\end{proof}

\section{ Appendix 1: the finite rank case}

In this appendix, we give a sketch of the proof of Theorem \ref{TheoDiffeoMN}, referring to \cite{GG2} for details. The mapping $\chi_N$  is of course well defined and smooth on $\mathcal V(2N)_{\rm gen}$. The explicit formula of $u$ in terms of $\chi_N(u)$ is obtained as before thanks to the compressed shift operator and it  proves that $\chi_N$ is one to one.
\subsection{A local diffeomorphism}
To prove that $\chi_N$ is a local diffeomorphism, we establish some identities on the Poisson Brackets. This set of identities imply that the differential of $\chi_N$ is of maximal rank so that $\chi_N$ is a local diffeomorphism. As a consequence, it is an open mapping.

 Let us first recall some basic definitions on Hamiltonian formalism.
Given a smooth real-valued function $F$ on a finite dimensional  symplectic manifold $(\mathcal M,\omega )$, the Hamiltonian vector field of $F$
is the vector field $X_F$ on $\mathcal M$ defined by
$$\forall m\in \mathcal M, \forall h\in T_m\mathcal M, dF(m).h=\omega (h, X_F(m))\ .$$
Given two smooth real valued functions $F,G$, the Poisson bracket of $F$ and $G$ is
$$\{ F,G\} =dG.X_F=\omega (X_F,X_G)\ .$$
The above identity is generalized to complex valued functions $F, G$ by $\C $-bilinearity.

To obtain that the image of the symplectoc form $\omega$ by $\chi_N$ is  given by Formula (\ref{ChiStarOmegaN}), one has to prove equivalently that $$(\chi_N)_*\omega=\sum_j\rho_jd\rho_j\wedge d\varphi_j+\sigma_jd\sigma_j\wedge d\theta_j$$
which includes the following identities.
\begin{proposition}\label{InvolActionAngle}
For any $j,k\in\{1,\dots,N\}$ , one has
\begin{eqnarray*}
\{\rho_j,\rho_k\}=\{\rho_j,\sigma_k\}=\{\sigma_j,\sigma_k\}=0\\
\{\rho_j,\varphi_k\}=\rho_j^{-1}\delta_{jk} ,\;  \{\sigma_j,\varphi_k\}=0\ ,\\
\{\rho_j,\theta_k\}=0,\; \{\sigma_j,\theta_k\}=\sigma_j^{-1}\delta_{jk}\ .
\end{eqnarray*}
\end{proposition}
In order to compute for instance $\{\sigma_j,\theta_k\}$ one has for instance to differentiate $\theta_k$ along the direction of $X_{\sigma_j}$. As  the expression of $X_{\sigma _j}$ is fairly complicated, we  use the "Szeg\" o hierarchy"  studied in \cite{GG}. More precisely, we use the generating function $J(x)=((I-xH_u^2)^{-1}(1)\vert 1)=1+\sum_{n=1}^\infty x^nJ_{2n}$.  In the sequel, we shall restrict ourselves to real values of $x$, so that $J(x)$ is a real valued function.\s
 We proved in \cite{GG} that the Hamiltonian flow associated to $J(x)$ as a function of $u$ admits a Lax pair involving the Hankel operator $H_u$. From this Lax pair, one can deduce easily a second one involving the operator $K_u$.
 
 \begin{theo} [The Szeg\"o hierarchy, \cite{GG}, Theorem 8.1 and Corollary 8]\label{Jdex}
Let $s>\frac 12$. The map $u\mapsto J(x)$ is smooth on $H^s_+$. Moreover, the equation $
\partial _tu=X_{J(x)}(u)$
implies $
\partial _tH_u=[B_{u}^x,H_u],$ or to $\partial _tK_u=[C_{u}^x,K_u],$ where $B_{u}^x$ and $C_{u}^x$ are skew-adjoint if $x$ is real. 
\end{theo}
\begin{remark}\label{actioncom}
As a direct consequence, the spectrum of $H_u$ as well as the spectrum of $K_u$ are conserved by the Hamiltonian flow of $J(x)$.  We infer that the Poisson brackets of $J(x)$ with $\rho _j$ or $\sigma _j$ are
zero, which implies, in view of Lemma \ref{Nu_j}, that the brackets of $\rho _k$ or $\sigma _\ell $ with $\rho _j$ or $\sigma _m$ are zero, hence it gives the first set of commutation properties stated in Proposition \ref{InvolActionAngle}.
\end{remark}
Using the Szeg\"o hierarchy, we can also compute the Poisson brackets of $J(x)$ with the angles. 
\begin{Lemma}\label{BracketsJxvarphi}
 $$\{J(x),\varphi_j\}=\frac 12 \frac{xJ(x)}{1-\rho _j^2x}\quad \{J(x),\theta_j\}=-\frac 12 \frac{xJ(x)}{1-\sigma _j^2x}.$$
\end{Lemma}
Using again the expression of $J(x)$, these commutation properties allow to obtain by identification of the polar parts the last commutation properties of Proposition \ref{InvolActionAngle}.

 To conclude that the image of the symplectic form $\omega$ is given by Formula (\ref{ChiStarOmegaN}), we need to establish the following remaining  commutation properties,
 $$\{ \varphi _j,\varphi _k\} =\{ \varphi _j,\theta _k\} =\{ \theta _j,\theta _k\} =0\ .$$
 In  \cite{GG2}, these identities are obtained as consequences of further calculations. Here we give a simpler argument. By Lemma \ref{InvolActionAngle}, one can write
 $$(\chi_N)_*\omega=\sum_j\rho_jd\rho_j\wedge d\varphi_j+\sigma_jd\sigma_j\wedge d\theta_j+\tilde \omega ,$$
where $\tilde \omega $ is a closed form depending only on variables $\rho _j,\sigma _m$. Consider the following real submanifold of $\mathcal V(2N)_{{\rm gen}}$, 
$$\Lambda _N=\{ u\in \mathcal V(2N)_{{\rm gen}}\, :\, \varphi _1=\dots =\varphi _N=\theta _1=\dots =\theta _N=0\} .$$
By formula (\ref{c_nDiffeo}), every element $u$ of $\Lambda _N$ has real Fourier coefficients. Consequently, $\omega =0$ on $\Lambda _N$. On the other hand, 
$(\chi_N)_*\omega=\tilde \omega $ on $\chi _N(\Lambda _N)$, and the $\rho _j,\sigma _m$ are coordinates on $\Lambda _N$. We conclude that $\tilde \omega =0$.

\subsection{Surjectivity: a compactness result}

As $\Xi_N$ is connected, it suffices to prove that $\chi_N$ is proper. Let us take a  sequence $(\zeta^{(p)})_p$ in $\Xi _N$ which converges to $\zeta \in \Xi _N$, and such that, for every $p$, there exists  $u_p\in {\mathcal V}(2N)_{{\rm gen}}$ with
$$\chi _N(u_p)=\zeta^{(p)}\ .$$
Since 
$$\Vert u_p\Vert _{VMO}=\Vert H_{u_p}\Vert=\max _{1\le j\le N} (\rho _j^{(p)})= \max _{1\le j\le N}(|\zeta_{2j-1}^{(p)}|)$$
$(u_p)$  is a bounded sequence in $VMO_+(\T)$.  Up to extracting a subsequence, we may assume that $(u_p)_{p\in\Z_+}$  converges weakly to some $u$ in $VMO_+(\T)$. At this stage we can appeal to Proposition 
\ref{compactness} and conclude that the convergence of $u_p$ to $u$ is strong and that
$$\rho _j(u)=|\zeta_{2j-1}|,\; \sigma _j(u)=|\zeta_{2j}|,\   j=1,\dots ,N $$
with $ \rho _j(u)=0$, $\sigma_j(u)=0$ if $j>N$.
Therefore  $u\in \mathcal V(2N)_{{\rm gen}}\ .$ This completes the proof of the surjectivity of $\chi _N$.

\section{Appendix 2: The boundedness of operator $A$.}

In this appendix, we prove the boundedness of operator $A$ defined by (\ref{A}) in Theorem \ref{TheoHomeo}. Of course, this boundedness follows from the theorem itself,
since it implies that $A$ is conjugated to the compressed shift operator. However, we found interesting to give a self-contained proof of this fact. We need the following two lemmas.

\begin{Lemma}
Let  $(\rho_j)_{j\ge 1}$ and $(\sigma_j)_{j\ge 1}$ be two sequences such that $$\rho_1^2>\sigma_1^2>\rho_2^2>\dots >\dots \to 0.$$
Then, the following quantities are well defined  and coincide respectively outside $\{\frac 1{\rho_j^2}\}_{j\ge 1}$ and $\{\frac 1{\sigma_j^2}\}_{j\ge 1}$
\begin{equation}\label{J(x)}
\prod_{j=1}^\infty \frac{1-x\sigma_j^2}{1-x\rho_j^2}=1+x\sum_{j=1}^\infty \frac{\nu_j^2\rho_j^2}{1-x\rho_j^2}
\end{equation}
\begin{equation}\label{1/J(x)}
\prod_{j=1}^\infty \frac{1-x\rho_j^2}{1-x\sigma_j^2}=1-x\left (C+\sum_{j=1}^\infty \frac{\kappa_j^2}{1-x\sigma_j^2}\right )
\end{equation}
where $$C=\left\{\begin{array}{ll}0\text{ if }\sum_{j=1}^\infty \nu_j^2<1\text{ or }\sum_{j=1}^\infty \nu_j^2=1\text{ and }\sum_{j=1}^\infty \nu_j^2\rho_j^{-2}=\infty\\
\\
(\sum_{j=1}^\infty \nu_j^2\rho_j^{-2})^{-1}\text{ if }\sum_{j=1}^\infty \nu_j^2=1\text{ and }\sum_{j=1}^\infty \nu_j^2\rho_j^{-2}<\infty
\end{array}\right.$$
Here the $\nu_j^2$'s are given by formula (\ref{nulambdamu}) and the $\kappa_j^2$'s by formula (\ref{kappa}).
\end{Lemma}
\begin{remark} 
Notice that formulae (\ref{J(x)}) and (\ref{1/J(x)}) can be interpreted in light of Theorem \ref{TheoHomeo}, as we did in Lemma \ref{Nu_j}.
More precisely, formula (\ref{J(x)}) gives the value of $J(x)=((I-xH_u^2)^{-1}(1)\vert 1)=1+x((I-xH_u^2)^{-1}u\vert u)$, while formula 
(\ref{1/J(x)}) gives the value of $1/J(x)=1-x((I-xK_u^2)^{-1}u\vert u)$. This provides an interpretation of constant $C$, as the contribution of
$\ker (K_u)\cap \overline {{\rm Ran}H_u)}$ in the expansion.
\end{remark}

\begin{proof}
We first consider finite sequences $(\rho_j)_{1\le j\le N}$ and $(\sigma_j)_{1\le j\le N}$ such that $\rho_1^2>\sigma_1^2>\rho_2^2>\dots >\sigma_N^2> 0.$ 
We claim that, for $x\notin \{\frac 1{\rho_j^2}\}_{j\ge 1}$
\begin{equation}\label{J_N}\prod_{j=1}^N \frac{1-x\sigma_j^2}{1-x\rho_j^2}=1+x\sum_{j=1}^N \frac{(\nu_j^{(N)})^2\rho_j^2}{1-x\rho_j^2}
\end{equation}
where 
$$(\nu_j^{(N)})^2= \left(1-\frac{\sigma_j^2}{\rho_j^2}\right)\prod_{k\ne j\atop 1\le k\le N}\frac{\rho_j^2-\sigma_k^2}{\rho_j^2-\rho_k^2}.$$
Indeed, both functions have the same poles, the same residue hence their difference is a polynomial. Moreover, this polynomial function tends to a constant at infinity, hence  is a constant. As both terms coincide at $x=0$, they coincide everywhere. 
It remains to let $N\to \infty$. The left hand side in (\ref{J_N}) tends to 
$$\prod_{j=1}^\infty \frac{1-x\sigma_j^2}{1-x\rho_j^2}$$ since this product converges in view of the assumption on the sequences $(\rho_j)$ and $(\sigma_j)$.

Let us consider the limit of the right hand side in Equality (\ref{J_N}).
Let $x$ tend to $-\infty$ in Equality (\ref{J_N}). We get
$$\prod_{j=1}^N\frac{\sigma_j^2}{\rho_j^2}=1-\sum_{j=1}^N(\nu_j^{(N)})^2.$$
In particular, $\sum_{j=1}^N(\nu_j^{(N)})^2$ is bounded by $1$ so $\sum_{j=1}^\infty \nu_j^2$ converges by Fatou's lemma.
For $x\le 0$,
$$\sum_{N\ge j\ge M}\frac{\rho_j^2(\nu_j^{(N)})^2}{1-x\rho_j^2}\le  \rho _M^2,$$
hence the series $\sum_{N\ge j\ge 1}\frac{\rho_j^2(\nu_j^{(N)})^2}{1-x\rho_j^2}$ is uniformly summable, and we infer
$$\sum_{j=1}^N \frac{\rho_j^2(\nu_j^{(N)})^2}{1-x\rho_j^2}\rightarrow \sum_{j=1}^\infty \frac{\rho_j^2\nu_j^2}{1-x\rho_j^2}\ .$$
It gives the first equality (\ref{J(x)}) for $x\le 0$ and for $x\notin \{\frac 1{\rho_j^2}\}_{j\ge 1}$ by analytic continuation.

For Equality (\ref{1/J(x)}), we do almost the same analysis. As before, as $N$ tends to $\infty$, 
$$\prod_{j=1}^N\frac{1-x\rho_j^2}{1-x\sigma_j^2}\to \prod_{j=1}^\infty\frac{1-x\rho_j^2}{1-x\sigma_j^2},\; x\notin \{\frac 1{\sigma_j^2}\}_{j\ge 1}.$$ On the other hand, for $x\notin \{\frac 1{\sigma_j^2}\}_{j\ge 1}$,
$$\prod_{j=1}^N\frac{1-x\rho_j^2}{1-x\sigma_j^2}=1-x\sum_{j=1}^N\frac{(\kappa_j^{(N)})^2}{1-x\sigma_j^2}$$
where 
$$(\kappa_j^{(N)})^2=(\rho_j^2-\sigma_j^2)\prod_{k\ne j\atop 1\le k\le N}\frac{\sigma_j^2-\rho_k^2}{\sigma_j^2-\sigma_k^2}.$$
We denote by $H_N$ the function defined  for $x\neq \sigma_j^{-2}$, $1\le j\le N$, as
\begin{equation}\label{H_N}H_N(x):=\sum_{j=1}^N\frac{(\kappa_j^{(N)})^2}{1-x\sigma_j^2}.\end{equation}
The preceding equality reads
\begin{equation}\label{H_Nprod}H_N(x)=\frac 1x\left(1-\prod_{j=1}^N\frac{1-x\rho_j^2}{1-x\sigma_j^2}\right).\end{equation}
 Using Formula (\ref{H_Nprod}) and the logarithmic derivative  of  $$\prod_{j=1}^N\frac{1-x\rho_j^2}{1-x\sigma_j^2}$$ at $x=0$, we get that $H_N(0)=\sum_{j=1}^N(\rho_j^2-\sigma_j^2)$. As by Formula (\ref{H_N}), $H_N(0)=\sum_{j=1}^N (\kappa_j^{(N)})^2$, we obtain
 $$\sum_{j=1}^N (\kappa_j^{(N)})^2=\sum_{j=1}^N(\rho_j^2-\sigma_j^2)$$ and this last sum is bounded independently of $N$ namely
$$\sum_{j=1}^N(\rho_j^2-\sigma_j^2)\le \sum_{j=1}^N(\rho_j^2-\rho_{j+1}^2)\le  \rho_1^2.$$ 
 Hence the sum $\sum \kappa_j^2$ converges by Fatou's lemma. 
 We use this property to justify the convergence of $H_N'(x)$. Indeed,  for $x\neq \sigma_j^{-2}$, $1\le j\le N$, $$H_N'(x)=\sum_{j=1}^N\frac{(\kappa_j^{(N)})^2\sigma_j^2}{(1-x\sigma_j^2)^2}.$$
So, a proof analogous as the one used before allows to show that
$$H_N'(x)\to \sum_{j=1}^\infty\frac{\kappa_j^2\sigma_j^2}{(1-x\sigma_j^2)^2}.$$
Furthermore, the convergence holds uniformly for $x\le 0$.
Therefore, on  one hand, as $N$ tends to $\infty$,
$$H_N(x)=\frac 1x\left (1-\prod_{j=1}^N\frac{1-x\rho_j^2}{1-x\sigma_j^2}\right )\to \frac 1x\left (1-\prod_{j=1}^\infty\frac{1-x\rho_j^2}{1-x\sigma_j^2}\right )$$ and on the other hand, as
$$H_N(x)=\int_{y}^xH'_N(t) dt +H_N(y)$$
we get at the limit as $N$ goes to $\infty$, for $x\le 0$, and hence everywhere by analytic continuation,
$$ \frac 1x\left (1-\prod_{j=1}^\infty\frac{1-x\rho_j^2}{1-x\sigma_j^2}\right )=\sum_{j=1}^\infty\frac{\kappa_j^2}{1-x\sigma_j^2}+C.$$
It remains to compute $C$ by taking the limit as $x$ goes to $-\infty$.
$$C=-\lim_{x\to -\infty}\frac 1x\prod_{j=1}^\infty \frac{1-x\rho_j^2}{1-x\sigma_j^2}=-\lim_{x\to -\infty}\frac 1{xJ(x)}$$
where 
$$
J(x):=\prod_{j=1}^\infty \frac{1-x\sigma_j^2}{1-x\rho_j^2}\ .$$

This limit has been computed in (\ref{limF(y)}) whenever $$\sum_{j=1}^\infty \nu_j^2=1\text{ and }\sum_{j=1}^\infty \frac{\nu_j^2}{\rho_j^2}<\infty$$ and is equal to  $$(\sum_{j=1}^\infty \frac{\nu_j^2}{\rho_j^2})^{-1}.$$
This calculation  easily extends to the other cases writing for $x<0$
\begin{eqnarray*}
J(x)&=&1+x\sum_{j=1}^\infty \frac{\nu_j^2\rho_j^2}{1-x\rho_j^2}\\
&=&1-\sum_{j=1}^\infty \nu_j^2+\sum_{j=1}^\infty \frac{\nu_j^2}{1-x\rho_j^2}.
\end{eqnarray*}
\end{proof}

\begin{corollary}For any $m\ge 1$, we have
 \begin{equation}\label{1}\sum_j\frac{\rho_j^2\nu_j^2}{\rho_j^2-\sigma_m^2}=1\end{equation}
 \begin{equation}\label{1kappa}\sum_j\frac{\kappa_j^2}{\rho_m^2-\sigma_j^2}+\frac C{\rho_m^2}=1\end{equation}
 \begin{equation}\label{2}\sum_j\frac{\rho_j^2\nu_j^2}{(\rho_j^2-\sigma_m^2)(\rho_j^2-\sigma_p^2)}=\frac 1{\kappa_m^2}\delta_{mp}\end{equation}
  \begin{equation}\label{2kappa}\sum_j\frac{\sigma_j^2\kappa_j^2}{(\sigma_j^2-\rho_m^2)(\sigma_j^2-\rho_p^2)}=\frac 1{\nu_m^2}\delta_{mp}-1\end{equation}

\end{corollary}

\begin{proof}
The  first two equalities (\ref{1}) and (\ref{1kappa}) are obtained by making $x=\frac 1{\sigma_m^2}$ and $x=\frac 1{\rho_m^2}$ respectively  in formula (\ref{J(x)}) and formula (\ref{1/J(x)}).
For equality (\ref{2}) in the case $m= p$, we first make the change of variable $y=1/x$ in formula (\ref{J(x)}) then differentiate both sides with respect to $y$ and make $y=\sigma_m^2$. 
Equality (\ref{2kappa}) in the case $m=p$ follows by differentiating equation (\ref{1/J(x)}) and making $x=\frac 1{\rho_m^2}$. 
Both equalities in the case $m\neq p$ follow directly respectively from equality (\ref{1}) and equality (\ref{1kappa}).
\end{proof}
\begin{Lemma}
Let $m$ be a fixed positive integer. Let $(\varphi_j)$ and $(\theta_m)$ be two sequences of elements of $\T$.
Denote by $A^{(m)}$ the rank $1$ operator  of matrix $$A^{(m)}=\left(\frac{\nu_j}{\rho_j^2-\sigma_m^2}\frac{\nu_k\rho_ke^{-i\varphi_k}}{\rho_k^2-\sigma_m^2}\sigma_m\kappa_m^2e^{-i\theta_m}\right)_{jk}\ .$$
 Then $A:=\sum_{m\ge 1} A^{(m)}$ defines a bounded operator on $\ell^2$ with $AA^*\le I$.
\end{Lemma}
\begin{proof}
First we notice that  $A^{(m)}$ satisfies $\Vert A^{(m)}\Vert\le 1$. This  follows from Cauchy-Schwarz inequality, formula (\ref{2}) and from the  estimate
\begin{eqnarray*}
\sum_j\frac{\nu_j^2}{(\rho_j^2-\sigma_m^2)^2}&=&-\frac 1{\sigma_m^2}\sum_j\frac{\nu_j^2}{\rho_j^2-\sigma_m^2}+\frac 1{\sigma_m^2}\sum_j\frac{\rho_j^2\nu_j^2}{(\rho_j^2-\sigma_m^2)^2}\\
&=&\frac 1{\sigma_m^2}\left(\frac{\sum_j\nu_j^2-1}{\sigma_m^2}+\frac 1{\kappa_m^2}\right)\\
&\le&\frac 1{\sigma_m^2\kappa_m^2},
\end{eqnarray*}
so that
\begin{eqnarray*}
\Vert A^{(m)}\Vert^2&\le& \sum_j\frac{\nu_j^2}{(\rho_j^2-\sigma_m^2)^2}\sum_k\frac{\nu_k^2\rho_k^2}{(\rho_k^2-\sigma_m^2)^2}\sigma_m^2\kappa_m^4\\
&\le &1.
\end{eqnarray*}
Let us consider the well defined operator $A^{(m)}(A^{(p)})^*$. An elementary calculation gives
\begin{eqnarray*}(A^{(m)}(A^{(p)})^*)_{jk}&=&\sum_{\ell} A^{(m)}_{j\ell}\overline{A_{k\ell}^{(p)}}\\
&=&\frac{\nu_j\nu_k}{(\rho_j^2-\sigma_m^2)(\rho_k^2-\sigma_m^2)}\sigma_m^2\kappa_m^2\delta_{mp}
\end{eqnarray*}
Taking the sum of both sides over $m$ and $p$, we get  by (\ref{2kappa}) that the sum converges and equals $\delta_{jk}-\nu_j\nu_k$. Consequently,  the sum of $A^{(m)}(A^{(p)})^*$ defines a bounded positive operator majorated by $I$  and coincides with the operator $AA^*$. It gives the boundedness of $A$ and completes the proof.

\end{proof}

\end{document}